\documentclass[11pt]{article}
\usepackage{amsmath,amsfonts,amssymb,amsthm}
\usepackage{enumerate}
\usepackage[pagebackref,colorlinks,citecolor=blue,linkcolor=blue]{hyperref}

\usepackage{geometry}
\numberwithin{equation}{section}
\theoremstyle{plain}
\newtheorem{theorem}[equation]{Theorem}

\newtheorem{lemma}[equation]{Lemma}
\newtheorem{proposition}[equation]{Proposition}

\theoremstyle{definition}
\newtheorem{definition}[equation]{Definition}

\newtheorem{example}[equation]{Example}
\newtheorem{remark}[equation]{Remark}

\numberwithin{equation}{section}

\newcommand{\R}{{\mathbb R}}
\newcommand{\N}{{\mathbb N}}

\newcommand{\Om}{\Omega}

\providecommand{\vint}[1]{\mathchoice
          {\mathop{\vrule width 5pt height 3 pt depth -2.5pt
                  \kern -9pt \kern 1pt\intop}\nolimits_{\kern -5pt{#1}}}
          {\mathop{\vrule width 5pt height 3 pt depth -2.6pt
                  \kern -6pt \intop}\nolimits_{\kern -3pt{#1}}}
          {\mathop{\vrule width 5pt height 3 pt depth -2.6pt
                  \kern -6pt \intop}\nolimits_{\kern -3pt{#1}}}
          {\mathop{\vrule width 5pt height 3 pt depth -2.6pt
                  \kern -6pt \intop}\nolimits_{\kern -3pt{#1}}}}

\newcommand{\eps}{\varepsilon}
\newcommand{\loc}{\mathrm{loc}}

\newcommand{\BV}{\mathrm{BV}}

\newcommand{\liploc}{\mathrm{Lip}_{\mathrm{loc}}}

\newcommand{\ch}{\text{\raise 1.3pt \hbox{$\chi$}\kern-0.2pt}}

\DeclareMathOperator{\capa}{Cap}
\DeclareMathOperator{\rcapa}{cap}

\DeclareMathOperator{\Lip}{Lip}

\DeclareMathOperator{\fint}{fine-int}

\begin{document}
\title{A sharp Leibniz rule for  $\BV$ functions
	in metric spaces
\footnote{{\bf 2010 Mathematics Subject Classification}: 30L99, 31E05, 26B30.
\hfill \break {\it Keywords\,}: function of bounded variation, Leibniz rule,
metric measure space, weak* convergence, quasi semicontinuity
}}
\author{Panu Lahti}
\maketitle

\begin{abstract}
We prove a Leibniz rule for $\BV$ functions in a complete metric space
that is equipped with a doubling measure and supports a Poincar\'e
inequality. Unlike in previous versions of the rule, we do not assume
the functions to be locally essentially bounded and the end
result does not involve a constant $C\ge 1$, and so our
result seems to be essentially the best possible.
In order to obtain the rule in such generality, we first study the weak* convergence
of the variation measure of $\BV$ functions, with
quasi semicontinuous test functions.
\end{abstract}

\section{Introduction}

The Leibniz rule for functions of bounded variation ($\BV$ functions)
says that if $u,v\in\BV(\R^n)\cap L^{\infty}(\R^n)$, then the
variation measures satisfy
\begin{equation}\label{eq:Euclidean Leibniz rule}
dD(uv) = \overline{u}\,d Dv+\overline{v}\,d Du,
\end{equation}
where $\overline{u},\overline{v}$ are the so-called precise
representatives of $u$ and $v$; see
\cite{Vol} or \cite[Section 4.6.4]{VH}.
More precisely, this result is proved in the above references with somewhat weaker
assumptions; in particular, the boundedness assumption can be weakened
to only one of the functions being locally (essentially) bounded.

In the past two decades, a theory of $\BV$ functions
as well as other topics in analysis has been developed in the
abstract setting of metric measure spaces. The standard assumptions in this setting
are that $(X,d,\mu)$ is a complete metric space equipped with a Borel regular,
\emph{doubling} outer measure $\mu$, and that $X$ supports a Poincar\'e inequality. See Section \ref{sec:preliminaries} for definitions.
In this setting, the following Leibniz rule for $\BV$ functions was proved in
\cite{KKST3}.

\begin{proposition}[{\cite[Proposition 4.2]{KKST3}}]\label{prop:KKST Leibniz rule}
	Let $u,v\in\BV(X)\cap L^{\infty}(X)$ be nonnegative functions.
	Then $uv\in\BV(X)\cap L^{\infty}(X)$ such that
	\[
	d\Vert D(uv)\Vert\le Cv^{\vee}\, d\Vert Du\Vert
	+Cu^{\vee}\, d\Vert Dv\Vert
	\]
	for some constant $C\ge 1$ that depends only on the doubling constant
	of the measure and the constants in the Poincar\'e inequality.
\end{proposition}

Note that in metric spaces, one cannot talk about the vector measure
$Du$, only the total variation $\Vert Du\Vert$. In the above Leibniz rule,
we see that again
the functions are assumed to be in $L^{\infty}(X)$.
Additionally, there is a multiplicative constant $C\ge 1$
that arises from the use of a \emph{discrete convolution} technique
in the proof of the Leibniz rule. This is a common technique in metric
space analysis, and sometimes the constant $C$ appearing in
an end result cannot be removed, see e.g.
\cite[Remark 4.7, Example 4.8]{HKLL}.
On the other hand, for the \emph{upper gradients} of
Newton-Sobolev functions (a generalization of Sobolev function to metric spaces),
one has the Leibniz rule
\[
g_{uv}\le ug_v+vg_u,
\]
which does not involve a constant $C$.
Thus it is natural to ask whether the constant $C$, as well as the
$L^{\infty}$-assumption, can be dropped from the $\BV$ Leibniz rule, and in this paper
we show that this is indeed the case. Our main result is the following.
\begin{theorem}\label{thm:Leibniz rule intro}
	Let $\Om\subset X$ be an open set and let $u,v\in L^1_{\loc}(\Om)$.
	Then
	\[
	\Vert D(uv)\Vert(\Om)\le \int_{\Om}|u|^{\vee}\,d\Vert Dv\Vert
	+\int_{\Om}|v|^{\vee}\,d\Vert Du\Vert.
	\]
\end{theorem}
Since neither of the functions $u,v$ is assumed to be
in $L^{\infty}_{\loc}(\Om)$, we do not automatically even have
$uv\in L^1_{\loc}(\Om)$,
but we are able to prove this assuming that the right-hand side is finite,
and then we also obtain $\Vert D(uv)\Vert(\Om)<\infty$.
Moreover, we do not assume the functions $u,v$ to be $\BV$ functions
even locally, so the measures $\Vert Du\Vert,\Vert Dv\Vert$ could be large;
it is only necessary that the two integrals are finite.
If this is the case, then we can obtain
\begin{equation}\label{eq:Leibniz rule measure version}
d\Vert D(uv)\Vert\le |u|^{\vee}\, d\Vert Dv\Vert
+|v|^{\vee}\, d\Vert Du\Vert,
\end{equation}
as measures on $\Om$,
improving on Proposition \ref{prop:KKST Leibniz rule} --- see Remark
\ref{rmk:Leibniz rule}.
In fact, due to our minimal assumptions, our result seems to give a slight improvement
on what is known even in Euclidean spaces.
On the other hand, in Example \ref{ex:pointwise representatives} we show
that unlike in Euclidean spaces, in metric spaces it is not possible
to replace the representatives $u^{\vee},v^{\vee}$ by $\overline{u},\overline{v}$,
and that equality may hold in \eqref{eq:Leibniz rule measure version}.
Thus our Leibniz rule appears to be
essentially the best possible
in every respect.

It can be said that the
Leibniz rules for BV functions that are found in the literature are already quite sufficient for
most applications, typically in the calculus of variations.
Thus perhaps the main interest of this paper is in the methods
that we employ.
Indeed, to prove the Leibniz rule in the above generality, we use
several rather strong tools and also develop a few new ones.
To avoid having to assume that $u,v$ are $\BV$ functions,
we use an extension property that relies on Federer's characterization
of sets of finite perimeter proved in \cite{L-Fed}.
The most significant effort is required in ensuring
that the constant $C$ does not appear in the end result.
For this, we use two tools: one is a result on the pointwise
convergence of $\BV$ functions given in \cite{L-SP}.
The other is a result on the weak* convergence of variation measures
in the case of \emph{quasi semicontinuous} test functions,
which we derive in Section \ref{sec:weak convergence} and which
is based on results in \cite{L-LSC}.

\section{Preliminaries}\label{sec:preliminaries}

In this section we present the necessary notation, definitions,
assumptions, and a few background results. 

Throughout this paper, $(X,d,\mu)$ is a complete metric space that is equip\-ped
with a metric $d$ and a Borel regular outer measure $\mu$ satisfying
a doubling property, meaning that
there exists a constant $C_d\ge 1$ such that
\[
0<\mu(B(x,2r))\le C_d\mu(B(x,r))<\infty
\]
for every ball $B(x,r):=\{y\in X:\,d(y,x)<r\}$.
When a property holds outside a set of $\mu$-measure zero, we say that it holds
almost everywhere, abbreviated a.e.

As a complete metric space equipped with a doubling measure, $X$ is proper,
that is, closed and bounded sets are compact.
All functions defined on $X$ or its subsets will take values in $[-\infty,\infty]$.
Given a $\mu$-measurable set $A\subset X$, we define $L^1_{\loc}(A)$ as the class
of functions $u$ on $A$
such that for every $x\in A$ there exists $r>0$ such that $u\in L^1(A\cap B(x,r))$.
Other local spaces of functions are defined analogously.
For an open set $\Omega\subset X$,
a function is in the class $L^1_{\loc}(\Omega)$ if and only if it is in $L^1(\Om')$ for
every open $\Omega'\Subset\Omega$.
Here $\Omega'\Subset\Omega$ means that $\overline{\Omega'}$ is a
compact subset of $\Omega$.

For any set $A\subset X$ and $0<R<\infty$, the restricted Hausdorff content
of codimension one is defined as
\[
\mathcal{H}_{R}(A):=\inf\left\{ \sum_{i=1}^{\infty}
\frac{\mu(B(x_{i},r_{i}))}{r_{i}}:\,A\subset\bigcup_{i=1}^{\infty}B(x_{i},r_{i}),\,r_{i}\le R\right\}.
\]
The codimension one Hausdorff measure of $A\subset X$ is then defined as
\[
\mathcal{H}(A):=\lim_{R\rightarrow 0}\mathcal{H}_{R}(A).
\]

By a curve we mean a nonconstant rectifiable continuous mapping from a compact interval of the real line
into $X$.
A nonnegative Borel function $g$ on $X$ is an upper gradient 
of a function $u$
on $X$ if for all curves $\gamma$, we have
\begin{equation}\label{eq:definition of upper gradient}
|u(x)-u(y)|\le \int_\gamma g\,ds,
\end{equation}
where $x$ and $y$ are the end points of $\gamma$
and the curve integral is defined by using an arc-length parametrization,
see \cite[Section 2]{HK} where upper gradients were originally introduced.
We interpret $|u(x)-u(y)|=\infty$ whenever  
at least one of $|u(x)|$, $|u(y)|$ is infinite.

We say that a family of curves $\Gamma$ is of zero $1$-modulus if there is a 
nonnegative Borel function $\rho\in L^1(X)$ such that 
for all curves $\gamma\in\Gamma$, the curve integral $\int_\gamma \rho\,ds$ is infinite.
A property is said to hold for $1$-almost every curve
if it fails only for a curve family with zero $1$-modulus. 
If $g$ is a nonnegative $\mu$-measurable function on $X$
and (\ref{eq:definition of upper gradient}) holds for $1$-almost every curve,
we say that $g$ is a $1$-weak upper gradient of $u$. 
By only considering curves $\gamma$ in $A\subset X$,
we can talk about a function $g$ being a ($1$-weak) upper gradient of $u$ in $A$.

Given a $\mu$-measurable set $H\subset X$, we let
\[
\Vert u\Vert_{N^{1,1}(H)}:=\Vert u\Vert_{L^1(H)}+\inf \Vert g\Vert_{L^1(H)},
\]
where the infimum is taken over all $1$-weak upper gradients $g$ of $u$ in $H$.
The substitute for the Sobolev space $W^{1,1}$ in the metric setting is the Newton-Sobolev space
\[
N^{1,1}(H):=\{u:\|u\|_{N^{1,1}(H)}<\infty\},
\]
which was first introduced in \cite{S}.
We understand a Newton-Sobolev function to be defined at every $x\in H$
(even though $\Vert \cdot\Vert_{N^{1,1}(H)}$ is then only a seminorm).
It is known that for any $u\in N_{\loc}^{1,1}(H)$ there exists a minimal $1$-weak
upper gradient of $u$ in $H$, always denoted by $g_{u}$, satisfying $g_{u}\le g$ 
a.e. in $H$, for any $1$-weak upper gradient $g\in L_{\loc}^{1}(H)$
of $u$ in $H$, see \cite[Theorem 2.25]{BB}.

We will assume throughout the paper that $X$ supports a $(1,1)$-Poincar\'e inequality,
meaning that there exist constants $C_P>0$ and $\lambda \ge 1$ such that for every
ball $B(x,r)$, every $u\in L^1_{\loc}(X)$,
and every upper gradient $g$ of $u$,
we have
\[
\vint{B(x,r)}|u-u_{B(x,r)}|\, d\mu 
\le C_P r\vint{B(x,\lambda r)}g\,d\mu,
\]
where 
\[
u_{B(x,r)}:=\vint{B(x,r)}u\,d\mu :=\frac 1{\mu(B(x,r))}\int_{B(x,r)}u\,d\mu.
\]

The $1$-capacity of a set $A\subset X$ is defined as
\[
\capa_1(A):=\inf \Vert u\Vert_{N^{1,1}(X)},
\]
where the infimum is taken over all functions $u\in N^{1,1}(X)$ such that $u\ge 1$ in $A$.

\begin{definition}
Let $H\subset X$.
	We say that a set $A\subset H$ is $1$-quasiopen with respect to $H$
	if for every $\eps>0$ there is an
	open set $G\subset X$ such that $\capa_1(G)<\eps$ and $(A\cup G)\cap H$ is open
	in the subspace topology of $H$.
	When $H=X$, we omit mention of it.

We say that a function $u$ is $1$-quasi (lower/upper semi-)continuous on $H$ if for every $\eps>0$ there exists an open set $G\subset X$ such that $\capa_1(G)<\eps$
and $u|_{H\setminus G}$ is real-valued (lower/upper semi-)continuous.
\end{definition}

It is a well-known fact that a Newton-Sobolev function $u\in N^{1,1}(\Om)$ is
$1$-quasicontinuous on an open set $\Om$,
see \cite[Theorem 1.1]{BBS} or \cite[Theorem 5.29]{BB}.

The variational $1$-capacity of a set $A\subset D$ with respect to a set $D\subset X$ is
defined as
\[
\rcapa_1(A,D):=\inf \int_{X} g_u\,d\mu,
\]
where the infimum is taken over functions $u\in N^{1,1}(X)$ such that $u\ge 1$ on $A$
and $u=0$ on $X\setminus D$.

Next we present the basic theory of functions
of bounded variation on metric spaces. This was first developed in
\cite{A1, M}; see also the monographs \cite{AFP, EvaG92, Fed, Giu84, Zie89} for the classical 
theory in Euclidean spaces.
We will always denote by $\Om$ an open subset of $X$.
Given a function $u\in L^1_{\loc}(\Om)$,
we define the total variation of $u$ in $\Om$ as
\[
\|Du\|(\Om):=\inf\left\{\liminf_{i\to\infty}\int_\Om g_{u_i}\,d\mu:\, u_i\in 
\liploc(\Om),\, u_i\to u\textrm{ in } L^1_{\loc}(\Om)\right\},
\]
where each $g_{u_i}$ is the minimal $1$-weak upper gradient of $u_i$
in $\Om$. If $u\notin L^1_{\loc}(\Om)$, we interpret $\Vert Du\Vert(\Om)=\infty$.
(In \cite{M}, local Lipschitz constants were used in place of upper gradients, but the theory
can be developed similarly with either definition.)
We say that a function $u\in L^1(\Om)$ is of bounded variation, 
and denote $u\in\BV(\Om)$, if $\|Du\|(\Om)<\infty$.
For an arbitrary set $A\subset X$, we define
\[
\|Du\|(A):=\inf\{\|Du\|(W):\, A\subset W,\,W\subset X
\text{ is open}\}.
\]
In general, we understand the expression $\Vert Du\Vert(A)<\infty$ to mean that
there exists some open set $\Om\supset A$ such that $u$ is defined on $\Om$ with $u\in L^1_{\loc}(\Om)$
and $\Vert Du\Vert(\Om)<\infty$.
\begin{theorem}[{\cite[Theorem 3.4]{M}}]\label{thm:variation measure property}
If $u\in L^1_{\loc}(\Om)$, then $\|Du\|(\cdot)$ is
a Borel measure on $\Omega$.
\end{theorem}
Note that this result does not require that $\Vert Du\Vert(\Om)<\infty$, as can
be seen from the proof in \cite{M}. We call $\Vert Du\Vert$ the variation measure
of $u$.
A $\mu$-measurable set $E\subset X$ is said to be of finite perimeter if $\|D\ch_E\|(X)<\infty$, where $\ch_E$ is the characteristic function of $E$.
The perimeter of $E$ in a set $A\subset X$ is also denoted by
\[
P(E,A):=\|D\ch_E\|(A).
\]

The measure-theoretic boundary $\partial^{*}E$ of a set $E\subset X$
is defined as the set of points
$x\in X$
at which both $E$ and its complement have strictly positive upper density, i.e.
\[
\limsup_{r\to 0}\frac{\mu(B(x,r)\cap E)}{\mu(B(x,r))}>0\quad
\textrm{and}\quad\limsup_{r\to 0}\frac{\mu(B(x,r)\setminus E)}{\mu(B(x,r))}>0.
\]
For an open set $\Omega\subset X$ and a $\mu$-measurable set $E\subset X$ with $P(E,\Omega)<\infty$, we know that for any Borel set $A\subset\Omega$,
\begin{equation}\label{eq:def of theta}
P(E,A)=\int_{\partial^{*}E\cap A}\theta_E\,d\mathcal H,
\end{equation}
where
$\theta_E\colon X\to [\alpha,C_d]$ with $\alpha=\alpha(C_d,C_P,\lambda)>0$, see \cite[Theorem 5.3]{A1} 
and \cite[Theorem 4.6]{AMP}.
The following coarea formula is given in \cite[Proposition 4.2]{M}:
if $u\in L^1_{\loc}(\Omega)$, then
\begin{equation}\label{eq:coarea}
\|Du\|(\Om)=\int_{-\infty}^{\infty}P(\{u>t\},\Om)\,dt.
\end{equation}
If $\Vert Du\Vert(\Om)<\infty$, then the formula holds with
$\Om$ replaced by any Borel set $A\subset \Om$.
From this combined with \eqref{eq:def of theta},
we obtain the absolute continuity
\begin{equation}\label{eq:absolute continuity of var measure wrt H}
\Vert Du\Vert\ll \mathcal H\quad\textrm{on }\Om.
\end{equation}
Since $\liploc(\Om)$ is dense in $N^{1,1}(\Om)$, see \cite[Theorem 5.47]{BB}, it follows that
\begin{equation}\label{eq:Sobolev subclass BV}
N^{1,1}(\Om)\subset \BV(\Om)\quad \textrm{with}\quad \Vert Du\Vert(\Om)
\le \int_\Om g_u\,d\mu\ \ 
\textrm{for every }u\in N^{1,1}(\Om).
\end{equation}

If we apply the $(1,1)$-Poincar\'e inequality to sequences
of approximating locally
Lipschitz functions in the definition of the total variation, we get
the following $\BV$ version:
for every ball $B(x,r)$ and every 
$u\in L^1_{\loc}(X)$, we have
\begin{equation}\label{eq:Poincare for BV}
\vint{B(x,r)}|u-u_{B(x,r)}|\,d\mu
\le C_P r\, \frac{\Vert Du\Vert (B(x,\lambda r))}{\mu(B(x,\lambda r))}.
\end{equation}

The lower and upper approximate limits of a function $u$ on $X$ are defined respectively by
\[
u^{\wedge}(x):
=\sup\left\{t\in\R:\,\lim_{r\to 0}\frac{\mu(B(x,r)\cap\{u<t\})}{\mu(B(x,r))}=0\right\}
\]
and
\[
u^{\vee}(x):
=\inf\left\{t\in\R:\,\lim_{r\to 0}\frac{\mu(B(x,r)\cap\{u>t\})}{\mu(B(x,r))}=0\right\}.
\]
It is straightforward to check that these are always Borel functions.
Unlike Newton-Sobolev functions, we understand $\BV$ functions to be
$\mu$-equivalence classes. To study fine properties, we need to
consider the pointwise representatives $u^{\wedge}$ and $u^{\vee}$.

Recall that Newton-Sobolev functions are quasicontinuous;
$\BV$ functions have the following partially analogous quasi-semicontinuity property, which was first proved in the
Euclidean setting in \cite[Theorem 2.5]{CDLP}.

\begin{proposition}\label{prop:quasisemicontinuity}
	Let  $u\in L^1_{\loc}(\Om)$ with
	$\Vert Du\Vert(\Om)<\infty$. Then
	$u^{\wedge}$ is $1$-quasi lower semicontinuous on $\Om$ and
	$u^{\vee}$ is $1$-quasi upper semicontinuous on $\Om$.
\end{proposition}
\begin{proof}
	This follows from \cite[Corollary 4.2]{L-SA} (which is based on
	\cite[Theorem 1.1]{LaSh}).
\end{proof}

For $D\subset\Om\subset X$, with $\Om$ again open, we define the class of $\BV$ functions with zero boundary values as
\begin{equation}\label{eq:definition of BV zero}
\BV_0(D,\Om):=
\left\{u|_{D}:\,u\in\BV(\Om),\ u^{\wedge}(x)=u^{\vee}(x)=0\textrm{ for }\mathcal H\textrm{-a.e. }x\in \Om\setminus D\right\}.
\end{equation}
This class was previously considered in \cite{L-ZB}.
It follows rather easily from the coarea formula \eqref{eq:coarea} that
for $u\in\BV_0(D,\Om)$, defining $u=0$ (a.e.) on $\Om\setminus D$, we have
\begin{equation}\label{eq:energy of BV zero class}
\Vert Du\Vert(\Om\setminus D)=0,
\end{equation}
see \cite[Proposition 3.14]{L-ZB}.

Next we define the fine topology in the case $p=1$.
\begin{definition}\label{def:1 fine topology}
We say that $A\subset X$ is $1$-thin at the point $x\in X$ if
\[
\lim_{r\to 0}r\frac{\rcapa_1(A\cap B(x,r),B(x,2r))}{\mu(B(x,r))}=0.
\]
We also say that a set $U\subset X$ is $1$-finely open if $X\setminus U$ is $1$-thin at every $x\in U$. Then we define the $1$-fine topology as the collection of $1$-finely open sets on $X$.

We denote the $1$-fine interior of a set $H\subset X$, i.e. the largest $1$-finely open set contained in $H$, by $\fint H$. We denote the $1$-fine closure of $H$,
i.e. the smallest $1$-finely closed set containing $H$, by $\overline{H}^1$.
\end{definition}

See \cite[Section 4]{L-FC} for a proof of the fact that the
$1$-fine topology is indeed a topology.
The following fact is given in \cite[Proposition 3.3]{L-Fed}:
\begin{equation}\label{eq:capacity of fine closure}
\capa_1(\overline{A}^1)=\capa_1(A)\quad\textrm{for any }A\subset X.
\end{equation}

\begin{theorem}[{\cite[Corollary 6.12]{L-CK}}]\label{thm:finely open is quasiopen and vice versa}
A set $U\subset X$ is $1$-quasiopen if and only if it is the union of a $1$-finely
open set and a $\mathcal H$-negligible set.
\end{theorem}

\emph{Throughout this paper we assume that $(X,d,\mu)$ is a complete metric space
	that is equipped with a doubling measure $\mu$ and supports a
	$(1,1)$-Poincar\'e inequality.}

\section{Weak* convergence of the variation measure}\label{sec:weak convergence}

In this section we prove some new results concerning
the weak* convergence of the variation measure.
First we collect some necessary existing results.

It follows almost directly from the definition of the total variation that
this quantity is
lower semicontinuous with respect to $L^1$-convergence in open sets.
We also have the following stronger fact.

\begin{theorem}[{\cite[Theorem 4.5]{L-LSC}}]\label{thm:lower semic in quasiopen sets}
	Let $U\subset X$ be a $1$-quasiopen set.
	If $\Vert Du\Vert(U)<\infty$ and $u_i\to u$ in $L^1_{\loc}(U)$, then
	\[
	\Vert Du\Vert(U)\le \liminf_{i\to\infty}\Vert Du_i\Vert(U).
	\]
\end{theorem}

Recall that we understand the expression
$\Vert Du\Vert(U)<\infty$ to mean that there is some open set $\Om\supset U$ such that
$u\in L^1_{\loc}(\Om)$ and $\Vert Du\Vert(\Om)<\infty$.

We also have the following.

\begin{theorem}[{\cite[Theorem 4.3]{L-LSC}}]\label{thm:characterization of total variational}
	Let $U\subset X$ be a $1$-quasiopen set. If $\Vert Du\Vert(U)<\infty$, then
	\[
	\Vert Du\Vert(U)=\inf \left\{\liminf_{i\to\infty}\int_{U}g_{u_i}\,d\mu,\,
	u_i\in N_{\loc}^{1,1}(U),\, u_i\to u\textrm{ in }L^1_{\loc}(U)\right\},
	\]
	where each $g_{u_i}$ is the minimal $1$-weak upper gradient of $u_i$ in $U$.
\end{theorem}

From this it follows that if $U$ is $1$-quasiopen,
$\Vert Du\Vert(U)<\infty$,
and $u_i\to u$ in $L^1_{\loc}(U)$, then
\begin{equation}\label{eq:upper gradient lsc in quasiopen set}
\Vert Du\Vert(U)\le \liminf_{i\to\infty}\int_U g_{u_i}\,d\mu,
\end{equation}
where each $g_{u_i}$ is the minimal $1$-weak upper gradient of $u_i$ in $U$;
naturally any other $1$-weak upper gradient can also be used.
Note that integrals over a $1$-quasiopen set $U$ make sense, since
every such set is $\mu$-measurable, see \cite[Lemma 9.3]{BB-OD}.

The variation measure is always absolutely continuous with respect to the $1$-capacity, in the following sense.

\begin{lemma}[{\cite[Lemma 3.8]{L-SA}}]\label{lem:variation measure and capacity}
Let $\Omega\subset X$ be an open set and
let $u\in L^1_{\loc}(\Omega)$ with $\Vert Du\Vert(\Omega)<\infty$. Then for every
$\eps>0$ there exists $\delta>0$ such that if $A\subset \Omega$ with $\capa_1 (A)<\delta$,
then $\Vert Du\Vert(A)<\eps$.
\end{lemma}

\begin{lemma}\label{lem:quasiopen sets are Du measurable}
Let $\Omega\subset X$ be an open set and
let $u\in L^1_{\loc}(\Omega)$ with $\Vert Du\Vert(\Omega)<\infty$. Then
every $1$-quasiopen subset of $\Om$ is $\Vert Du\Vert$-measurable.
\end{lemma}
\begin{proof}
Let $U\subset \Om$ be $1$-quasiopen.
For each $j\in\N$, choose an open set $G_j\subset\Om$ such that
$U\cup G_j$ is open and $\capa_1(G_j)\to 0$ as $j\to\infty$.
Let $H:=\bigcap_{j=1}^{\infty}(U\cup G_j)$.
Then $U\subset H$ and
\[
\Vert Du\Vert(H\setminus U)\le \Vert Du\Vert(G_j)\to 0\quad\textrm{as }j\to\infty
\]
by Lemma \ref{lem:variation measure and capacity}.
The set $H$ is $\Vert Du\Vert$-measurable since it is a Borel set, and then
also $U$ is $\Vert Du\Vert$-measurable.
\end{proof}

\begin{lemma}\label{lem:quasiopen with respect to quasiopen}
Let $U\subset X$ be $1$-quasiopen and let $V\subset U$ be $1$-quasiopen with respect to
$U$. Then $V$ is $1$-quasiopen.
\end{lemma}
\begin{proof}
Let $\eps>0$. We find open sets $G_1,G_2\subset X$ such that $\capa_1(G_1)<\eps/2$,
$\capa_1(G_2)<\eps/2$, $(V\cup G_1)\cap U$ is open in the subspace topology of
$U$, and $U\cup G_2$ is open.
Then clearly $V\cup G_1\cup G_2$ is open.
\end{proof}

It is known that if $\Om\subset X$ is an open set,
$u_i\to u$ in $L^1_{\loc}(\Om)$, and
\[
\lim_{i\to\infty}\Vert Du_i\Vert(\Om) =\Vert Du\Vert(\Om)<\infty,
\]
then $\Vert Du_i\Vert\to \Vert Du\Vert$ weakly* as measures
in $\Om$.
Using Theorem \ref{thm:lower semic in quasiopen sets} we will show that in fact the
measures converge in a stronger topology, namely in the dual of
\emph{quasicontinuous} functions instead of continuous ones.
This result may naturally be of independent interest and so we prove
it in somewhat greater generality than is necessary for our purposes.

\begin{proposition}\label{prop:lower semic in quasiopen sets}
	Let $U\subset X$ be $1$-quasiopen and let
	$\Vert Du\Vert(U)<\infty$. If $u_i\to u$ in $L^1_{\loc}(U)$
	such that
	\[
	\Vert Du\Vert(U)= \lim_{i\to\infty}\Vert Du_i\Vert(U),
	\]
	then
	\[
	\int_U \eta\,d\Vert Du\Vert= \lim_{i\to\infty}\int_U \eta\,d\Vert Du_i\Vert
	\]
	for every bounded $1$-quasicontinuous function $\eta$ on $U$.
	
	Similarly, if $u_i\to u$ in $L^1_{\loc}(U)$
	such that
	\[
	\Vert Du\Vert(U)= \lim_{i\to\infty}\int_U g_{u_i}\,d\mu,
	\]
	where each $g_{u_i}$ is the minimal $1$-weak upper gradient of $u_i$ in $U$, then
	\[
	\int_U \eta\,d\Vert Du\Vert= \lim_{i\to\infty}\int_U \eta g_{u_i}\,d\mu
	\]
	for every bounded $1$-quasicontinuous function $\eta$ on $U$.
\end{proposition}

\begin{proof}
	We follow an argument that can be found e.g. in \cite[Proposition 1.80]{AFP}.
	Let $\eta$ be a $1$-quasicontinuous bounded function on $U$.
	By replacing $\eta$ by $a \eta+b$ for suitable $a,b\in\R$, we can assume that
	$0\le \eta\le 1$. Note that every super-level set $\{\eta>t\}$, $t\in\R$, is $1$-quasiopen
	with respect to $U$, and thus $1$-quasiopen by Lemma \ref{lem:quasiopen with respect to quasiopen}.
	By Lemma \ref{lem:quasiopen sets are Du measurable}, $1$-quasiopen sets are $\Vert Du\Vert$-measurable, and so the integrals in the formulation of  the proposition make sense.
	Let $\rho$ be a nonnegative real-valued $1$-quasicontinuous function on $U$.
	By Cavalieri's principle, Fatou's lemma, and Theorem \ref{thm:lower semic in quasiopen sets} we have
	\begin{equation}\label{eq:use of Fatous lemma}
	\begin{split}
	\liminf_{i\to\infty}\int_U \rho\,d\Vert Du_i\Vert
	&=\liminf_{i\to\infty}\int_0^{\infty}\Vert Du_i\Vert(\{\rho>t\})\,dt\\
	&\ge\int_0^{\infty}\liminf_{i\to\infty}\Vert Du_i\Vert(\{\rho>t\})\,dt\\
	&\ge\int_0^{\infty}\Vert Du\Vert(\{\rho>t\})\,dt\\
	&= \int_U \rho\,d\Vert Du\Vert.
	\end{split}
	\end{equation}
	It is easy to check that if $(a_i)$ and $(b_i)$ are sequences of numbers such that
	\[
	\liminf_{i\to\infty}a_i\ge a,\quad\liminf_{i\to\infty}b_i\ge b,\quad\textrm{and }
	\lim_{i\to\infty} (a_i+b_i)= a+b,
	\]
	then $\lim_{i\to\infty}a_i=a$ and $\lim_{i\to\infty}b_i=b$. Choosing
	\begin{align*}
	& a_i=\int_U \eta\,d\Vert Du_i\Vert,\quad a=\int_U \eta\,d\Vert Du\Vert,\\
	& b_i=\int_U (1-\eta)\,d\Vert Du_i\Vert,\quad b=\int_U (1-\eta)\,d\Vert Du\Vert,
	\end{align*}
	and using the fact that $\Vert Du_i\Vert(U)\to \Vert Du\Vert(U)$ as well as
	\eqref{eq:use of Fatous lemma} with the choices $\rho=\eta$ and $\rho=1-\eta$,
	we obtain the first claim. The second claim if proved analogously, using
	\eqref{eq:upper gradient lsc in quasiopen set} instead of
	Theorem \ref{thm:lower semic in quasiopen sets}.
\end{proof}

Now we get the following result which we will use in the sequel.

\begin{proposition}\label{prop:weak convergence with quasisemicontinuous test function}
Let $U\subset X$ be $1$-quasiopen.
If $\Vert Du\Vert(U)<\infty$ and $u_i\to u$ in $L^1_{\loc}(U)$
such that
\begin{equation}\label{eq:sequence yielding Du U}
\Vert Du\Vert(U)= \lim_{i\to\infty}\int_U g_{u_i}\,d\mu,
\end{equation}
where each $g_{u_i}$ is the minimal $1$-weak upper gradient of $u_i$ in $U$, then
\[
\int_U \eta\,d\Vert Du\Vert\ge \limsup_{i\to\infty}\int_U \eta g_{u_i}\,d\mu
\]
for every nonnegative bounded $1$-quasi upper semicontinuous function $\eta$ on $U$.
\end{proposition}

\begin{proof}
Take $M>0$ such that $0\le \eta\le M$ on $U$.
For each $j\in\N$ we find an open set $G_j\subset X$ such that $\capa_1(G_j)<1/j$ and
$\eta|_{U\setminus G_j}$ is upper semicontinuous.
Then each $\eta_j:=\ch_{U\setminus G_j} \eta$ is upper semicontinuous on $U$. 
Let
\[
\eta_{j,k}(x):=\sup\{\eta_j(y)-kd(y,x):\,y\in U\},\quad x\in U,\ k\in\N.
\]
It is easy to check that $\eta_j \le \eta_{j,k}\le M$, $\eta_{j,k}\in\Lip(U)$,
and $\eta_{j,k}\searrow\eta_j$
pointwise as $k\to\infty$.
Fix $\eps>0$.
By Lemma \ref{lem:variation measure and capacity} we find $0<\delta<\eps$ such that
whenever $A\subset U$ with $\capa_1(A)<\delta$, then $\Vert Du\Vert(A)<\eps$.
Choose $j\in\N$ such that $\capa(G_j)<\delta$.
Then,
using Lebesgue's dominated convergence theorem, choose $k\in\N$ such that
$ \int_U \eta_j\,d\Vert Du\Vert\ge \int_U \eta_{j,k}\,d\Vert Du\Vert-\eps$.
By Theorem \ref{thm:finely open is quasiopen and vice versa} we know that
$X\setminus \overline {G_j}^1$ is a $1$-quasiopen set, and
then so is $U\setminus \overline{G_j}^1$,
since it is easy to check that the intersection of two $1$-quasiopen
sets is $1$-quasiopen (this fact is also proved in \cite[Lemma 2.3]{Fug}). Thus
by \eqref{eq:sequence yielding Du U} and \eqref{eq:upper gradient lsc in quasiopen set}
we have that
\begin{equation}\label{eq:upper semicontinuity in Gj}
\begin{split}
\Vert Du\Vert(U \cap \overline{G_j}^1)
&=\Vert Du\Vert(U)-\Vert Du\Vert(U \setminus \overline{G_j}^1)\\
&\ge\lim_{i\to\infty}\int_{U} g_{u_i}\,d\mu-\liminf_{i\to\infty}\int_{U \setminus \overline{G_j}^1} g_{u_i}\,d\mu\\
&\ge\limsup_{i\to\infty}\int_{U\cap G_j} g_{u_i}\,d\mu.
\end{split}
\end{equation}
Moreover, by \eqref{eq:capacity of fine closure},
$\capa_1(U\cap\overline{G_j}^1)\le \capa_1(G_j)<\delta$ and then $\Vert Du\Vert(U\cap\overline{G_j}^1)<\eps$.
Now
\begin{align*}
\int_U \eta\,d\Vert Du\Vert
&\ge \int_U \eta_j\,d\Vert Du\Vert\\
&\ge \int_U \eta_{j,k}\,d\Vert Du\Vert-\eps\\
&= \lim_{i\to\infty}\int_U \eta_{j,k} g_{u_i}\,d\mu-\eps\quad\textrm{by Proposition \ref{prop:lower semic in quasiopen sets}}\\
&\ge \limsup_{i\to\infty}\int_U \eta_{j} g_{u_i}\,d\mu-\eps\\
&\ge \limsup_{i\to\infty}\int_U \eta g_{u_i}\,d\mu
-M\limsup_{i\to\infty}\int_{U\cap G_j} g_{u_i}\,d\mu-\eps\\
&\ge \limsup_{i\to\infty}\int_U \eta g_{u_i}\,d\mu
-M\Vert Du\Vert(U\cap\overline{G_j}^1)-\eps\quad\textrm{by }\eqref{eq:upper semicontinuity in Gj}\\
&\ge \limsup_{i\to\infty}\int_U \eta g_{u_i}\,d\mu
-M\eps-\eps.
\end{align*}
Since $\eps>0$ was arbitrary, we have the result.
\end{proof}

\section{Proof of the Leibniz rule}

In this section we prove the Leibniz rule for $\BV$ functions,
Theorem \ref{thm:Leibniz rule intro}.
Again, $\Om$ always denotes an open set.

First we note that the total variation is lower semicontinuous not only with respect to
$L_{\loc}^1$-convergence, but also pointwise convergence.

\begin{proposition}\label{prop:lower semicontinuity pointwise}
Let $u$ be finite a.e. on $\Om$ and let
$(u_i)\subset L^1_{\loc}(\Om)$ such that $u_i\to u$ a.e. on $\Om$.
Then
	\[
	\Vert Du\Vert(\Om)\le \liminf_{i\to\infty}\Vert Du_i\Vert(\Om).
	\]
\end{proposition}
In particular, if the right-hand side is finite, then $u\in L^1_{\loc}(\Om)$.
\begin{proof}
	We can assume that the right-hand side is finite.
	Let $B(x,r)$ be a ball such that $B(x,\lambda r)\subset\Om$.
	By the Poincar\'e inequality \eqref{eq:Poincare for BV},
	\begin{align*}
	\infty
	>\liminf_{i\to\infty}\Vert Du_i\Vert(B(x,\lambda r))
	&\ge \liminf_{i\to\infty}\Vert D(u_i)_+\Vert(B(x,\lambda r))\\
	&\ge \frac{1}{C_P r}\liminf_{i\to\infty}\int_{B(x,r)}|(u_i)_+ - ((u_i)_+)_{B(x,r)}|\,d\mu\\
	&\ge \frac{1}{C_P r}\int_{B(x,r)}\liminf_{i\to\infty}|(u_i)_+ - ((u_i)_+)_{B(x,r)}|\,d\mu
	\end{align*}
	by Fatou's lemma. Since $\lim_{i\to\infty}(u_i)_+$ exists and is finite a.e., we conclude
	that $\liminf_{i\to\infty}((u_i)_+)_{B(x,r)}$ is finite. By another application of Fatou's lemma, it follows that $(u_+)_{B(x,r)}$ is finite. We conclude that
	$u_+ \in L^1_{\loc}(\Om)$. Similarly we show that $u_-\in L^1_{\loc}(\Om)$.
	In conclusion, $u\in L^1_{\loc}(\Om)$.
	If $M>0$ and $u_M:=\min\{M,\max\{-M,u\}\}$, we have $(u_i)_M\to u_M$ in $L^1_{\loc}(\Om)$
	by Lebesgue's dominated convergence theorem, and so
	\[
	\Vert D u_M\Vert(\Om)\le \liminf_{i\to\infty}\Vert D(u_i)_M\Vert(\Om)\le
	\liminf_{i\to\infty}\Vert D u_i\Vert(\Om).
	\]
	On the other hand, since we now know that $u\in L^1_{\loc}(\Om)$,
	we also know that $u_M\to u$ in $L^1_{\loc}(\Om)$ as $M\to\infty$, and so
	\[
	\Vert D u\Vert(\Om)\le \liminf_{M\to\infty}\Vert Du_M\Vert(\Om).
	\]
	The result follows.
\end{proof}

Now we turn to the proof of the Leibniz rule. As we recall from the introduction,
the Leibniz rule for Newton-Sobolev functions is a standard result also in metric spaces.
We begin by noting that the rule is easy to extend to the case where
\emph{one} of the functions is a $\BV$ function.
\begin{lemma}\label{lem:Leibniz rule for BV consequence}
Let $u\in L^{\infty}(\Om)$ with $\Vert Du\Vert(\Om)<\infty$ and let
$\eta\in N^{1,1}(\Om)\cap L^{\infty}(\Om)$. Then
	\[
	\Vert D(\eta u)\Vert(\Om)\le \int_{\Om}|\eta| \,d\Vert Du\Vert
	+\int_{\Om}|u|g_{\eta}\,d\mu.
	\]
\end{lemma}

\begin{proof}
By the definition of the total variation, we find a
sequence $(u_i)\subset \liploc(\Om)$ such that $u_i\to u$ in $L_{\loc}^1(\Om)$
and
\[
\lim_{i\to\infty}\int_{\Om}g_{u_i}\,d\mu= \Vert Du\Vert(\Om).
\]
By truncating if necessary, we can assume that
$\Vert u_i\Vert_{L^{\infty}(\Om)}\le \Vert u\Vert_{L^{\infty}(\Om)}$,
and by passing to a subsequence (not relabeled)
we can assume that $u_i\to u$ a.e. in $\Om$.
Also, $\eta u_i\to \eta u$ in $L^1_{\loc}(\Om)$, and so by lower semicontinuity
and the Leibniz rule for Newton-Sobolev functions (see \cite[Theorem 2.15]{BB}),
\begin{align*}
\Vert D(\eta u)\Vert(\Om)
&\le \liminf_{i\to\infty}\int_{\Om}g_{\eta u_i}\,d\mu\\
&\le \liminf_{i\to\infty}\left(\int_{\Om}|\eta|g_{u_i}\,d\mu
+\int_{\Om}|u_i|g_{\eta} \,d\mu\right)\\
&= \int_{\Om}|\eta|\,d\Vert Du\Vert
+\int_{\Om}|u|g_{\eta} \,d\mu
\end{align*}
by the second part of Proposition \ref{prop:lower semic in quasiopen sets}
(recall that the Newton-Sobolev function $\eta$ is quasicontinuous by e.g. 
\cite[Theorem 5.29]{BB})
and by Lebesgue's dominated convergence theorem.
\end{proof}

This first step of proving the $\BV$ Leibniz rule was essentially the same as in \cite{KKST3}.
However, to handle the case where \emph{both} functions are $\BV$ functions
(or even more generally locally integrable functions),
we will rely on the theory of Section \ref{sec:weak convergence} as well as the
following results.

\begin{theorem}[{\cite[Theorem 3.2]{L-SP}}]\label{thm:strict to pointwise convergence}
Let $u_i,u\in \BV(\Omega)$ such that $u_i\to u$ in $L^1(\Omega)$ and $\Vert Du_i\Vert(\Omega)\to \Vert Du\Vert(\Omega)$. Then there exists a subsequence (not relabeled) such that for $\mathcal H$-a.e. $x\in \Omega$,
\[
u^{\wedge}(x)\le\liminf_{i\to\infty}u_i^{\wedge}(x)
\le
\limsup_{i\to\infty}u_i^{\vee}(x)\le u^{\vee}(x).
\]
\end{theorem}

Next we note that Federer's characterization of sets of finite perimeter holds also
in metric spaces.

\begin{theorem}[{\cite[Theorem 1.1]{L-Fedchar}}]\label{thm:characterization}
	Let $\Omega\subset X$ be open, let $E\subset X$ be $\mu$-measurable, and suppose that
	$\mathcal H(\partial^*E \cap \Omega)<\infty$. Then $P(E,\Omega)<\infty$.
\end{theorem}

Recall the definition of the class $\BV_0(D,\Om)$ from \eqref{eq:definition of BV zero}.
The following result and its proof are similar to \cite[Theorem 6.1]{LaSh2}, which
was originally based on \cite[Theorem 1.1]{KKST}.

\begin{theorem}\label{thm:characterization of BV0}
Let $W\subset \Om\subset X$ be open sets and let $u\in\BV(W)$ such that
\begin{equation}\label{eq:Lebesgue type boundary condition}
\lim_{r\to 0}\frac{1}{\mu(B(x,r))}\int_{B(x,r)\cap W}|u|\,d\mu=0
\end{equation}
for $\mathcal H$-a.e. $x\in \Om\cap \partial W$. Then $u\in\BV_0(W,\Om)$.
\end{theorem}
\begin{proof}
Define the zero extension $u:=0$ (a.e.) on $\Om\setminus W$.
Fix $x\in\Om\cap \partial W$ such 
that \eqref{eq:Lebesgue type boundary condition} holds.
If $t<0$, then
\begin{align*}
\limsup_{r\to 0}\frac{\mu(B(x,r)\setminus \{u>t\})}{\mu(B(x,r))}
&=\limsup_{r\to 0}\frac{\mu(B(x,r)\cap W\setminus \{u>t\})}{\mu(B(x,r))}\\
  &\le \limsup_{r\to 0} \frac{1}{|t|\mu(B(x,r))}\int_{B(x,r)\cap W}|u|\,d\mu=0.
\end{align*}
If $t>0$, then
\begin{align*}
\limsup_{r\to 0}\frac{\mu(B(x,r)\cap \{u>t\})}{\mu(B(x,r))}
&=\limsup_{r\to 0}\frac{\mu(B(x,r)\cap W\cap \{u>t\})}{\mu(B(x,r))}\\
  &\le \limsup_{r\to 0} \frac{1}{|t|\mu(B(x,r))}\int_{B(x,r)\cap W}|u|\,d\mu=0.
\end{align*}
In both cases it follows that $x\notin \partial^*\{u>t\}$.
Clearly this is true also for
every $x\in \Om\setminus \overline{W}$.
In conclusion, $\mathcal H(\partial^*\{u>t\}\cap \Om\setminus W)=0$
for every $t\neq 0$.
By the coarea formula \eqref{eq:coarea} we know that $P(\{u>t\},W)<\infty$
for a.e. $t\in\R$. For such $t\neq 0$, by \eqref{eq:def of theta} we have
\[
\mathcal H(\partial^{*}\{u>t\}\cap \Om)=\mathcal H(\partial^{*}\{u>t\}\cap W)\le 
\alpha^{-1} P(\{u>t\},W)<\infty.
\]
By Theorem \ref{thm:characterization} it follows that $P(\{u>t\},\Om)<\infty$.
Since $\mathcal H(\partial^*\{u>t\}\cap \Om\setminus W)=0$,
by \eqref{eq:def of theta} we have $P(\{u>t\},\Om\setminus W)=0$.
Thus
\[
\int_{-\infty}^\infty P(\{u>t\},\Om)\, dt=\int_{-\infty}^\infty P(\{u>t\},W)\, dt
=\Vert Du\Vert(W),
\]
where the last equality follows from the coarea formula \eqref{eq:coarea}.
Since now $\int_{-\infty}^\infty P(\{u>t\},\Om)\, dt<\infty$, by another
application of the coarea formula we find that $\Vert Du\Vert(\Om)<\infty$.
From \eqref{eq:Lebesgue type boundary condition} we easily get
$u^{\wedge}(x)=u^{\vee}(x)=0$ for $\mathcal H$-a.e. $x\in\Om\setminus W$.
In conclusion, $u\in\BV_0(W,\Om)$.
\end{proof}

Now we prove our main result, which we restate here.

\begin{theorem}\label{thm:Leibniz rule}
Let $\Om\subset X$ be an open set and let $u,v\in L^1_{\loc}(\Om)$.
Then
\begin{equation}\label{eq:Leibniz rule for 2 functions}
\Vert D(uv)\Vert(\Om)\le \int_{\Om}|u|^{\vee}\,d\Vert Dv\Vert
+\int_{\Om}|v|^{\vee}\,d\Vert Du\Vert.
\end{equation}
\end{theorem}

\begin{remark}\label{rmk:Leibniz rule}
Note that since $|u|^{\vee},|v|^{\vee}$ are Borel functions
and $\Vert Du\Vert,\Vert Dv\Vert$ are Borel measures by Theorem
\ref{thm:variation measure property}, the integrals are always well defined.

Given functions $u,v\in L^1_{\loc}(\Om)$, it is of course not always true that
$uv\in L^1_{\loc}(\Om)$, but implicit in the theorem is the fact that if the right-hand
side is finite, then necessarily $uv\in L^1_{\loc}(\Om)$.
Suppose this is the case.
Theorem \ref{thm:variation measure property} implies that
$\Vert D(uv)\Vert$, $\Vert Dv\Vert$, and $\Vert Du\Vert$
are all Borel measures on $\Om$, and since $|u|^{\vee}$ and $|v|^{\vee}$ are
Borel functions, it is a standard result (see e.g. \cite[Theorem 1.29]{Rud})
that $|u|^{\vee}\,d\Vert Dv\Vert$ and $|v|^{\vee}\,d\Vert Du\Vert$
are Borel measures on $\Om$.
They are also of finite mass (note that $\Vert Du\Vert$ and $\Vert Dv\Vert$
themselves might not be even locally finite), and thus by e.g. \cite[Proposition 1.43]{AFP}
we know that the measure of Borel sets can be approximated from the outside
by open sets.
Inequality \eqref{eq:Leibniz rule for 2 functions} holds of course
also with $\Omega$ replaced by any open set $W\subset \Om$,
and then
\[
d\Vert D(uv)\Vert \le |u|^{\vee}\,d\Vert Dv\Vert
+|v|^{\vee}\,d\Vert Du\Vert
\]
as Borel measures on $\Om$.
\end{remark}

\begin{proof}[Proof of Theorem \ref{thm:Leibniz rule}]
First assume that $\Om$ is bounded, that $u,v\in L^{\infty}(\Om)$, and that
$\Vert Du\Vert(\Om)<\infty$ and $\Vert Dv\Vert(\Om)<\infty$.
Take a sequence $(u_i)\subset \liploc(\Om)$
such that $u_i\to u$ in $L_{\loc}^1(\Om)$ and
\[
\lim_{i\to\infty}\int_{\Om}g_{u_i}\,d\mu= \Vert Du\Vert(\Om).
\]
We can assume that $\Vert u_i\Vert_{L^{\infty}(\Om)}\le \Vert u\Vert_{L^{\infty}(\Om)}$
for all $i\in\N$.
Under our assumptions, we can in fact also assume that
$(u_i)\subset N^{1,1}(\Om)$ with
$u_i\to u$ in $L^1(\Om)$.
By \eqref{eq:Sobolev subclass BV} and by the lower semicontinuity of the total variation,
it follows that also
$\lim_{i\to\infty}\Vert Du_i\Vert(\Om)=\Vert Du\Vert(\Om)$.
By Lemma \ref{lem:Leibniz rule for BV consequence} we have for all $i\in\N$
\begin{equation}\label{eq:estimate for D ui v}
\begin{split}
\Vert D(u_i v)\Vert(\Om)
\le \int_{\Om}|u_i|\,d\Vert Dv\Vert+\int_{\Om}|v|g_{u_i}\,d\mu
=\int_{\Om}|u_i|\,d\Vert Dv\Vert+\int_{\Om}|v|^{\vee}g_{u_i}\,d\mu.
\end{split}
\end{equation}
We pass to a subsequence of $(u_i)$ (not relabeled) for which the conclusion of
Theorem \ref{thm:strict to pointwise convergence} holds.
Note that now also for $\mathcal H$-a.e. $x\in\Om$,
\begin{equation}\label{eq:pointwise convergence for minus case}
\limsup_{i\to\infty}(-u_i)^{\vee}(x)=\limsup_{i\to\infty}-u_i^{\wedge}(x)
=-\liminf_{i\to\infty}u_i^{\wedge}(x)\le -u^{\wedge}(x)
=(-u)^{\vee}(x).
\end{equation}
Since the $u_i$ are continuous functions, of course $u_i=u_i^{\wedge}=u_i^{\vee}$.
Now by
Theorem \ref{thm:strict to pointwise convergence} we have for $\mathcal H$-a.e. $x\in \Om$, and thus also for $\Vert Dv\Vert$-a.e. $x\in \Om$
(recall \eqref{eq:absolute continuity of var measure wrt H})
\[
\limsup_{i\to\infty}u_i(x)=\limsup_{i\to\infty}u_i^{\vee}(x)
\le u^{\vee}(x)\le |u|^{\vee}(x)
\]
and
\[
\limsup_{i\to\infty}(-u_i)(x)=\limsup_{i\to\infty}(-u_i)^{\vee}(x)
\overset{\eqref{eq:pointwise convergence for minus case}}{\le} (-u)^{\vee}(x)\le |u|^{\vee}(x),
\]
so that $\limsup_{i\to\infty}|u_i|(x)\le |u|^{\vee}(x)$.
Recall that $|u|^{\vee}$ is a Borel function and thus $\Vert Dv\Vert$-measurable.
Now we have by Fatou's lemma
\[
\limsup_{i\to\infty}\int_{\Om}|u_i|\,d\Vert Dv\Vert
\le \int_{\Om}\limsup_{i\to\infty}|u_i|\,d\Vert Dv\Vert
\le \int_{\Om}|u|^{\vee}\,d\Vert Dv\Vert.
\]
Since $\Vert Dv\Vert(\Om)<\infty$, clearly also $\Vert D|v|\Vert(\Om)<\infty$, and so
by Proposition \ref{prop:quasisemicontinuity}, $|v|^{\vee}$ is a bounded
$1$-quasi upper semicontinuous function on $\Om$.
Thus we have by Proposition \ref{prop:weak convergence with quasisemicontinuous test function}
\[
\limsup_{i\to\infty}\int_{\Om}|v|^{\vee}g_{u_i}\,d\mu
\le \int_{\Om}|v|^{\vee}\,d\Vert Du\Vert.
\]
Since $u_iv\to uv$ in $L^1(\Om)$, by lower semicontinuity
and by \eqref{eq:estimate for D ui v} we now have
\[
\Vert D(u v)\Vert(\Om)\le\liminf_{i\to\infty}\Vert D(u_i v)\Vert(\Om)
\le \int_{\Om}|u|^{\vee}\,d\Vert Dv\Vert+ \int_{\Om}|v|^{\vee}\,d\Vert Du\Vert,
\]
which is the desired result.

Now we drop the assumption that $\Vert Du\Vert(\Om)<\infty$
and $\Vert Dv\Vert(\Om)<\infty$.
We can assume that the right-hand side of \eqref{eq:Leibniz rule for 2 functions}
is finite. Let
\[
A_k:=\{x\in\Om:\,|u|^{\vee}(x)>1/k\textrm{ and }|v|^{\vee}(x)>1/k\},\quad k\in\N.
\]
Then necessarily $\Vert Du\Vert(A_k)<\infty$ and $\Vert Dv\Vert(A_k)<\infty$,
and so for some open set $W_k$
with $A_k\subset W_k\subset \Om$ we have $\Vert Du\Vert(W_k)<\infty$
and $\Vert Dv\Vert(W_k)<\infty$.
Then the theorem holds in $W_k$, that is,
\[
\Vert D(u v)\Vert(W_k)\le
\int_{W_k}|u|^{\vee}\,d\Vert Dv\Vert+ \int_{W_k}|v|^{\vee}\,d\Vert Du\Vert
\le
\int_{\Om}|u|^{\vee}\,d\Vert Dv\Vert+ \int_{\Om}|v|^{\vee}\,d\Vert Du\Vert.
\]
Letting $k\to\infty$, by Theorem \ref{thm:variation measure property} we get
for $W:=\bigcup_{k=1}^{\infty}W_k$
\begin{equation}\label{eq:D uv in W}
\Vert D(u v)\Vert(W)
\le \int_{\Om}|u|^{\vee}\,d\Vert Dv\Vert
+\int_{\Om}|v|^{\vee}\,d\Vert Du\Vert.
\end{equation}
Also, $(uv)|_W\in L^1(W)$ since $W$ is bounded, and so $(uv)|_W\in\BV(W)$.
Since $W\supset \{|u|^{\vee}>0\}\cap \{|v|^{\vee}>0\}$
and since $u,v\in L^{\infty}(\Om)$, we have
\[
\lim_{r\to 0}\frac{1}{\mu(B(x,r))}\int_{B(x,r)\cap W}|uv|\,d\mu=0
\]
for all $x\in \partial W\cap \Om$.
By Theorem \ref{thm:characterization of BV0} we find that $(uv)|_W\in\BV_0(W,\Om)$.
We have $uv=0$ a.e. on $\Om\setminus W$
by Lebesgue's differentiation theorem
(see e.g. \cite[Chapter 1]{Hei}).
Thus by \eqref{eq:energy of BV zero class} and \eqref{eq:D uv in W},
\[
\Vert D(uv)\Vert(\Om)=\Vert D(uv)\Vert(W)\le \int_{\Om}|u|^{\vee}\,d\Vert Dv\Vert
+\int_{\Om}|v|^{\vee}\,d\Vert Du\Vert.
\]

Next we drop the assumption that $\Om$ is bounded. For a fixed point $x\in X$, we have
by Theorem \ref{thm:variation measure property}
\begin{align*}
\Vert D(u v)\Vert(\Om)
&=\lim_{R\to\infty}\Vert D(u v)\Vert(\Om\cap B(x,R))\\
&\le \limsup_{R\to\infty}\left(\int_{\Om\cap B(x,R)}|u|^{\vee}\,d\Vert Dv\Vert
+\int_{\Om\cap B(x,R)}|v|^{\vee}\,d\Vert Du\Vert\right)\\
&\le \int_{\Om}|u|^{\vee}\,d\Vert Dv\Vert
+\int_{\Om}|v|^{\vee}\,d\Vert Du\Vert.
\end{align*}

Finally we drop the assumption $u,v\in L^{\infty}(\Om)$. Let
\[
u_M:=\min\{M,\max\{-M,u\}\},\quad M>0.
\]
By the above, we have
\begin{align*}
\Vert D(u_M v_M)\Vert(\Om)
&\le \int_{\Om}|u_M|^{\vee}\,d\Vert Dv_M\Vert
+\int_{\Om}|v_M|^{\vee}\,d\Vert Du_M\Vert\\
&\le \int_{\Om}|u|^{\vee}\,d\Vert Dv\Vert
+\int_{\Om}|v|^{\vee}\,d\Vert Du\Vert.
\end{align*}
Now we can use the lower semicontinuity of Proposition \ref{prop:lower semicontinuity pointwise} to get
\[
\Vert D(u v)\Vert(\Om)\le \liminf_{M\to\infty}\Vert D(u_M v_M)\Vert(\Om)
\le \int_{\Om}|u|^{\vee}\,d\Vert Dv\Vert
+\int_{\Om}|v|^{\vee}\,d\Vert Du\Vert.
\]
\end{proof}

\begin{example}\label{ex:pointwise representatives}
Recall that in Euclidean spaces we have the Leibniz rule \eqref{eq:Euclidean Leibniz rule},
which for nonnegative $u,v$ yields the scalar version
\[
d\Vert D(uv)\Vert \le \overline{u} \,d\Vert Dv\Vert+\overline{v}\,d \Vert Du\Vert.
\]
Here $\overline{u}(x):=\limsup_{r\to 0}\vint{B(x,r)}u\,d\mathcal L^n$,
where $\mathcal L^n$ is the $n$-dimensional Lebesgue measure.
In metric spaces this version
of the Leibniz rule does not hold;
in \cite[Example 4.3]{KKST3} (only in the arxiv version) the following counterexample
was given: equip $\R^2$ with the weighted Lebesgue measure $d\mu:=w\,d\mathcal L^2$,
where $w:=2-\ch_{B(0,1)}$ and the origin $(0,0)$ is denoted by $0$.
Let $u:=v:=\ch_{B(0,1)}$. Then $\overline{u}=\overline{v}=1/3$
on $\partial B(0,1)$, and it follows that
\[
d\Vert D(uv)\Vert=d\Vert Du\Vert>\frac 23 d\Vert Du\Vert
=\overline{v}\,d\Vert Du\Vert+\overline{u}\,d\Vert Dv\Vert.
\] 
On the other hand, sometimes one also defines $\overline{u}:=(u^{\wedge}+u^{\vee})/2$. With this definition we would have $\overline{u}=\overline{v}=1/2$
on $\partial B(0,1)$ and then
\[
d\Vert D(uv)\Vert=\overline{v}\,d\Vert Du\Vert+\overline{u}\,d\Vert Dv\Vert.
\]
Thus with this definition of the representative $\overline{u}$,
we seem to need a different type of counterexample,
which we construct as follows.
Consider the space
\[
X:=\{x=(x_1,x_2) \in \R^2:\,x_1=0\text{ or }x_2=0\}
\]
consisting of the two coordinate axes.
Equip this space with the Euclidean metric inherited from $\R^2$, and let $\mu$ be the 1-dimensional Hausdorff measure. 
It is straightforward to check that this measure
is doubling and supports a $(1,1)$-Poincar\'e inequality.
Let
\[
u:=\ch_{X}-\ch_{\{x_2>0\}}\in \BV_{\loc}(X)
\]
and
\[
v:=\ch_{X}-\ch_{\{x_1>0\}}\in \BV_{\loc}(X).
\]
For brevity, denote the origin $(0,0)$ by $0$.
It is straightforward to check that
\[
\Vert Du\Vert(X)=\Vert Du\Vert(0)=\Vert Dv\Vert(X)=\Vert Dv\Vert(0)=1,
\]
and that
$\Vert D(uv)\Vert(X)=\Vert D(uv)\Vert(0)\le 2$.
To see that in fact equality holds, take a sequence
$(u_i)\subset \liploc(X)$ such that
$u_i\to uv$ in $L^1_{\loc}(X)$.
Passing to a subsequence (not relabeled) we have also $u_i\to uv$ a.e. in $X$.
Thus we find $0<t<1$ such that for the points $x_1:=(t,0)$, $x_2:=(0,t)$,
$x_3:=(-t,0)$, and $x_4:=(0,-t)$ we have
$u_i(x_1)\to 0$, $u_i(x_2)\to 0$, $u_i(x_3)\to 1$, and $u_i(x_4)\to 1$
as $i\to \infty$.
Note that in this one-dimensional setting, each pair $(u_i,g_{u_i})$
satisfies the upper gradient inequality for \emph{every} curve in the space.
Let $\eps>0$. Now
\begin{align*}
\int_{B(0,1)}g_{u_i}\,d\mu
&\ge |u_i(x_1)-u_i(0)|+|u_i(x_2)-u_i(0)|+|u_i(x_3)-u_i(0)|+|u_i(x_4)-u_i(0)|\\
&\ge 2|u_i(0)|+2|1-u_i(0)|-\eps\quad\textrm{for large }i\in\N\\
&\ge 2-\eps.
\end{align*}
Since $\eps>0$ was arbitrary, it follows that $\Vert D(uv)\Vert(X)\ge 2$
and so in fact equality holds.
 
In conclusion,
\[
\Vert D(uv)\Vert(0)=2>1=\tfrac 12 \cdot 1+\tfrac 12 \cdot 1
=\overline{u}(0) \Vert Dv\Vert(0)+\overline{v}(0)\Vert Du\Vert(0).
\]
Here we have in fact exactly
\[
d\Vert D(uv)\Vert
=u^{\vee} \,d\Vert Dv\Vert+v^{\vee}\,d\Vert Du\Vert,
\]
demonstrating in this respect the sharpness of our Leibniz rule.
\end{example}

\begin{remark}
Consider inequality \eqref{eq:estimate for D ui v} in the proof of
Theorem \ref{thm:Leibniz rule},
\[
\Vert D(u_i v)\Vert(\Om)
\le \int_{\Om}|u_i|\,d\Vert Dv\Vert+\int_{\Om}|v|^{\vee}g_{u_i}\,d\mu.
\]
In the proof of the Leibniz rule given in \cite{KKST3}
(recall Proposition \ref{prop:KKST Leibniz rule}),
the functions $u_i$ were taken to be discrete convolution
approximations of $u$, because such approximations and their upper
gradients can be described by explicit formulas which
enables analysis of limiting behavior. However,
this produces the constant $C\ge 1$ in the end result.
We are able to avoid this constant by exploiting the convenient
way in which two BV functions ``pair up'' on the right-hand side
of \eqref{eq:estimate for D ui v}: the pointwise convergence of
$(u_i)$ is $\Vert Dv\Vert$-almost everywhere, and the weak*
convergence of $g_{u_i}\,d\mu$ is in a strong enough topology that
$|v|^{\vee}$ can act as a test function.
\end{remark}

\noindent Address:\\

\noindent \noindent Institut f\"ur Mathematik\\
Universit\"at Augsburg\\
Universit\"atsstr. 14\\
86159 Augsburg, Germany\\
E-mail: {\tt panu.lahti@math.uni-augsburg.de}

\end{document}